\documentclass[11pt]{article}

\usepackage{times}
\usepackage{amsthm}
\usepackage{amsmath}
\usepackage{amssymb}
\usepackage{mathrsfs}
\usepackage{pb-diagram}
\usepackage[all]{xy}
\usepackage{color}
\usepackage{comment}

\usepackage{hyperref}

\renewcommand\eqref[1]{(\ref{#1})} 

\def\A{{\mathcal A}}
\def\o{\omega}

\def\th{\theta}

\def\1{{\boldsymbol 1}}

\def\H{\mathcal H}
\def\K{\mathcal K}
\def\R{\mathbb{R}}

\def\e{{\sf e}}
\def\g{{\mathfrak g}}

\def\im{\mathop{\mathsf{Im}}\nolimits}

\def\({\left(}
\def\[{\left[}
\def\){\right)}
\def\]{\right]}

\def\si{\sigma}

\def\Aut{\mathfrak{Aut}}

\def\G{{\sf G}}

\def\p{\parallel}
\def\<{\langle}
\def\>{\rangle}

\newtheorem{Theorem}{Theorem}[section]
\newtheorem{Remark}[Theorem]{Remark}

\newtheorem{Definition}[Theorem]{Definition}
\newtheorem{Definition-Proposition}[Theorem]{Definition-Proposition}

\numberwithin{equation}{section}

\begin{document}


\title{Spectral Analysis for Perturbed Operators on Carnot Groups}

\date{\today}

\author{Marius M\u antoiu 
}
\date{\small}
\maketitle \vspace{-1cm}


\begin{abstract}
Let $\G$ be a Carnot group of homogeneous dimension $M$ and $\Delta$ its horizontal sublaplacian. For $\alpha\in(0,M)$ we show that operators of the form $H_\alpha:=(-\Delta)^\alpha+V$ have no singular spectrum, under generous assumptions on the multiplication operator $V$. The proof is based on commutator methods and Hardy inequalities.
\end{abstract}

\section{Introduction}\label{duci}

The purpose of this article is to show that certain operators on Carnot groups have purely absolutely continuous spectrum. These operators are essentially perturbations of sublaplacians by multiplication operators. We stress that the functions defining these multiplication operators are not supposed to decay at infinity. 

\smallskip
Let $H$ be a self-adjoint densely defined operator in a (complex separable) Hilbert space $\H$ with spectral measure $E_H$\,. If $u\in\H$\,, then 
\begin{equation*}\label{spmeas}
\mu_H^u(\cdot):=\<E_H(\cdot)u,u\>=\,\p\! E_H(\cdot)u\!\p^2
\end{equation*} 
defines a positive Borel measure on $\R$ with total mass $\p\!u\!\p^2$\,. We say that
\begin{itemize}
\item
the vector $u$ belongs to {\it the pure point space of} $H$ (and write $u\in\H_H^{\rm p}$) if $\mu^u_H$ is atomic,
\item
the vector $u$ belongs to {\it the absolutely continuous space of} $H$ (and write $u\in\H_H^{\rm ac}$) if $\mu^u_H$ is absolutely continuous with respect to the Lebesgue measure,
\item
the vector $u$ belongs to {\it the singularly continuous space of} $H$ (and write $u\in\H_H^{\rm sc}$) if $\mu^u_H$ is singular with respect to the Lebesgue measure and atomless.
\end{itemize}
One has direct sum decompositions
\begin{equation*}\label{decomposition}
\H=\H_H^{\rm p}\oplus\H_H^{\rm ac}\oplus\H_H^{\rm sc}=\H_H^{\rm p}\oplus\H_H^{\rm c}=\H_H^{\rm ac}\oplus\H_H^{\rm s}\,,
\end{equation*}
where $\H_H^{\rm c}:=\H_H^{\rm ac}\oplus\H_H^{\rm sc}$ is {\it the continuous space} and $\H_H^{\rm s}:=\H_H^{\rm p}\oplus\H_H^{\rm sc}$ is {\it the singular space} (with respect to $H$). We note that $\H_H^{\rm p}$ is the closed subspace generated by the eigenvectors of $H$\,. For our operators we are going to prove that $\H_H^{\rm s}=\{0\}$\,; this means that $H$ has no eigenvalues and no singularly continuous spectrum.
For the importance of determining the spectral types of a self-adjoint operator, for examples and applications (to scattering theory for example) we refer to \cite{ReSi}; actually the literature before or after \cite{ReSi} is huge.

\smallskip
The $m$-dimensional Carnot group $\G$ is stratified, so it possesses interesting {\it horizontal sublaplacians} $\Delta:=\sum_{j=1}^{m_1}X_j^2$, constructed with the part of a basis $\{X_1,\dots,X_{m_1},\dots,X_m\}$ of its Lie algebra $\g$ corresponding to the first vector subspace of its grading. The study of such operators is an important topic in Functional and Harmonic Analysis \cite{FR,FS,Ri}. For $\alpha>0$ wet us set $L^\alpha:=(-\Delta)^{\alpha/2}$\,; it is a self-adjoint unbounded positive operator in $\H:=L^2(\G)$\,. We consider perturbations $H_\alpha:=L^\alpha+V$ by multiplication operators (also called potentials). If $\G=\R^n$ is the simple Abelian case and $\alpha=2$\,, these are usually called {\it Schr\"odinger operators} and play a central role in Quantum Mechanics.

\smallskip
For such operators we show that $\H_H^{\rm s}=\{0\}$\,, under some assumptions on the potential $V$\,. A main result is described in section \ref{traci} and some other versions are exposed in section \ref{umilire}.
Our proof relies on a method in spectral theory involving positivity of the commutator of the studied operator with an auxiliary operator. It comes from \cite{BKM} (see also \cite{BM,CHM}), it is an extension to unbounded operators of a classical result of Kato and Putnam \cite{Ka,Pu,ReSi} and it is different from the well-known Mourre theory \cite{Mou}. It has been applied for various problems in spectral analysis in \cite{IM,MR,MRT,MT}. Schr\"odinger operators (Abelian $\G$) have been treated in \cite{BKM}. This method is reviewed in subsection \ref{vitalis} and applied to our case in subsection \ref{intrebare}. To apply the method, we rely heavily on dilations and on Hardy inequalities.


\section{The main result}\label{traci}

Let $\G$ be a Carnot group of step $r$\,, with unit $\e$ and Haar measure $dx$\,. Its Lie algebra can be written as a direct sum of vector subspaces
\begin{equation}\label{lie}
\g=\mathfrak v_1\oplus\dots\oplus \mathfrak v_r\,,
\end{equation}
where $\,[\mathfrak v_1,\mathfrak v_k]=\mathfrak v_{k+1}$ for every $k=1,\dots,r-1$ and $\,[\mathfrak v_1,\mathfrak v_r]=\{0\}$\,. Obviously $\G$ is a stratified group \cite{FR,FS}, hence (connected and simply connected) nilpotent and thus the exponential map ${\sf exp}:\mathfrak g\to\G$ is a diffeomorphism with inverse denoted by ${\sf log}$\,. We set notations for the dimensions
\begin{equation*}\label{dim}
m_k:=\dim\mathfrak v_k\,,\quad m:=\dim\g=m_1+m_2+\dots+m_r\,.
\end{equation*}
and for {\it the homogeneous dimension} 
\begin{equation*}\label{homdim}
M:=m_1+2m_2\dots+rm_r\,.
\end{equation*}

We fix a basis $\{X_1,\dots,X_m\}$ of $\g$ such that the first $m_1$ vectors $\{X_1,\dots,X_{m_1}\}$ form a basis of $\mathfrak v_1$\,. It can be used to define coordinates on the group $\G\ni x\to({\sf x}_1,\dots,{\sf x}_m)\in\R^m$ by the formula $\,x={\sf exp}\big[\sum_{j=1}^m{\sf x}_jX_j\big]$ (i.e. $({\sf x}_1,\dots,{\sf x}_m)$ are the coordinates of ${\sf log}\,x\in\g$ with respect to this basis). For a function $V:\G\to\R$ which is smooth in $\G\setminus\{\e\}$ we set
\begin{equation*}\label{prut}
(\mathbf EV)(x):=\sum_{j=1}^m\nu_j{\sf x}_j(X_j V)(x)\,,
\end{equation*}
where $\,\nu_j:=k$ if $\,m_{k-1}< j\le m_k$ (we set $n_0=0$ for convenience). One could say that $\mathbf E V$ is {\it the radial derivative of the function $V$ with respect to the natural dilations} (cf. subsect. \ref{vitanidis}).

\smallskip
Looking at the elements of the Lie algebra as derivations in spaces of functions on $\G$\,, one has {\it the horizontal gradient} $\,\nabla:=(X_1,\dots,X_{m_1})$ and {\it the horizontal sublaplacian}
\begin{equation*}\label{subL}
\Delta:=-\sum_{j=1}^{m_1}X_j^* X_j=\sum_{j=1}^{m_1}X_j^2.
\end{equation*} 
It is known that $-\Delta$ is a self-adjoint positive operator in $L^2(\G)$\,; its domain is the Sobolev space $\H^2(\G)$\,. Then $\,L^\alpha:=(-\Delta)^{\alpha/2}$ is also a self-adjoint positive operator; its domain is the Sobolev space $\H^\alpha(\G)$\,, defined as in \cite{FR,FS,Ri}. Note that, with our convention, one has $-\Delta=L^2$\,.

\smallskip
Let us choose a homogeneous quasi-norm $\lfloor\cdot\rfloor:\G\to[0,\infty)$ on $\G$\,. It is known \cite{CCR} that, if $\beta\in(0,M/2)$\,, the operator $\lfloor x\rfloor^{-\beta}(-\Delta)^{-\beta/2}$, first defined on $C_0^\infty(\G)$\,, extends to a bounded operator $T$ in $L^2(\G)$\,. Let us denote by $\kappa_{2\beta}$ its operator norm (it is not known in general).

\begin{Theorem}\label{maine}
Let $\alpha\in(0,M)$ and $\,V:\G\to\mathbb R$ a bounded function of class $C^2(\G\setminus\{\e\})$\,,  with $\mathbf EV$ bounded, such that
\begin{enumerate}
\item[(i)] $\,|\mathbf E V(x)|\le\big(\alpha\kappa^{-2}_{\alpha}-\epsilon\big)\lfloor x\rfloor^{-\alpha}$\, for some $\epsilon>0$ and for every $x\ne\e$\,,
\item[(ii)] $\,|\mathbf E(\mathbf E V)(x)|\le \,B\,\lfloor x\rfloor^{-\alpha}\,$ for some positive constant $B$ and for every $x\ne\e$\,.
\end{enumerate}
Then the self-adjoint operator $H_\alpha:=L^\alpha+V$ has purely absolutely spectrum.
\end{Theorem}

\begin{Remark}\label{simplest}
{\rm As it will be clear from the proof, it is not necessary for $V$ and $\mathbf EV$ to be bounded; some relative boundedness condition would be enough. We allowed this simplifications, being mainly interested in the behavior of the potential at infinity.
}
\end{Remark}

\begin{Remark}\label{zimplest}
{\rm The simplest Carnot groups are the Abelian groups $\G=\R^m$; in this case $M=m$. There is no loss of generality in choosing the (elliptic) Laplace operator $\Delta$ corresponding to the canonical basis of $\R^m$. One can treat $H_\alpha:=(-\Delta)^{\alpha/2}+V$ for $\alpha<m$\,. Results in this simple case can be found in \cite{BKM,BM}. If $\alpha=2$ one needs $m\ge 3$\,; taking $\lfloor x\rfloor=|x|$\,, the usual Euclidean norm, one has $\kappa_2=\frac{2}{m-2}$\,. For $m=1,2$ the results fail: in this case, if $V$ is non-trivial smooth negative and compactly supported, then for every $\eta>0$ the operator $-\Delta+\eta V$ has at least a (negative) eigenvalue. Of course, Hardy's estimate also fails in this case.
}
\end{Remark}

\begin{Remark}\label{incep}
{\rm We describe now the behavior at infinity required by the conditions (i) and (ii), using the $\R$-dilations of the Carnot group (see subsection \ref{vitanidis}).
Let us say that the function $V$ is {\it admissible} if it satisfies the hypohesis of the Theorems. If this is the case and $U$ is a homogeneous function of degree $0$ (i.e. $U\circ{\sf dil}_t=U$ for every $t\in\R$)\,, then $V+U$ is also admissible, because $\mathbf EU=0$\,. So it is clear that for a potential $V$ to be admissible it is not at all necessary to decay at infinity.

On the other hand, assume that a function $F\in C^\infty(\G\setminus\{\e\})$ satisfies $|(\mathbf E F)(x)|\le\frac{D}{\lfloor x\rfloor^\alpha}$ for every $x\ne\e$ and set $\Omega:=\{x\in\G\mid \lfloor x\rfloor=1\}$ for the unit quasi-sphere associated to the homogeneous quasi-norm. For $r>0$ let us set $r\cdot \o:={\sf dil}_{\log r}(\o)$\,. Then it is easy to show that for every $\omega\in\Omega$ there exists $f(\omega):=\lim_{r\to\infty}F(r\cdot \o)$ and one has $|F(r\cdot \o)-f(\o)|\le\frac{D}{\alpha}r^{-\alpha}$. The starting point is to write for $t_0,t_1\in\R$
$$
F\big(e^{t_1}\cdot\o\big)-F\big(e^{t_0}\cdot\o\big)=\int_{t_0}^{t_1}\!\frac{d}{ds}F\big(e^{s}\cdot\o\big)ds=\int_{t_0}^{t_1}\!(\mathbf EF)\big(e^{s}\cdot\o\big)ds\,;
$$
then one has to use the assumption and the homogeneity of the quasi-norm. This has to be applied to $F=V$ and to $F=\mathbf EV$\,. 
The required velocity of convergence towards the radial limits increases with $\alpha$\,; the convergence can be very slow if $\alpha$ is close to zero. Note, however, that we only described necessary conditions; (i) and (ii) contain more information about the oscillations of the function $V$.
}
\end{Remark}

\section{Proof of the main result}\label{firtanunus}

\subsection{The method of the weakly conjugate operator}\label{vitalis}

We present briefly a procedure \cite{BKM,BM}, called {\it the method of the weakly conjugate operator}, to show that that a given self-adjoint operator has no singular spectrum. The method also yields useful estimates on the resolvent family of the operator; this will be treated in a broader setting elsewhere.

\smallskip
Let $H$ be a self-adjoint densely defined operator in the Hilbert space $(\H,\<\cdot,\cdot\>)$\,, with domain ${\rm Dom}(H):=\mathcal G^2$, form domain $\mathcal G^1$ and spectral measure $E_H$\,. In $\mathcal G^1$ and $\mathcal G^2$ one considers the natural graph norms. By duality and interpolation one gets a scale of Hilbert spaces $\big\{\mathcal G^\si\mid \si\in[-2,2]\big\}$ where, for example, $\mathcal G^{-2}$ is the topological anti-dual of $\mathcal G^2$ (formed of anti-linear continuous functionals on $\mathcal G^2$\,). For positive $\si$\,, the norm in $\mathcal G^\si$ is $\p\!\cdot\!\p_\si:=\p\!({\rm Id}+|H|^2)^{\si/2}\cdot\!\p$\,. In particular, one has the continuous dense embeddings
\begin{equation*}\label{onehas}
\mathcal G^2\hookrightarrow\mathcal G^1\hookrightarrow\H\equiv\mathcal G^0\hookrightarrow\mathcal G^{-1}\hookrightarrow\mathcal G^{-2}.
\end{equation*}
Besides the scalar products inducing the new norms, one has various convenient extensions of the initial scalar product; the most useful will be $\mathcal G^1\times\mathcal G^{-1}\ni(u,w)\to\<u,w\>:=w(u)\in\mathbb C$\,; see \eqref{sancheliu} for instance. By duality and interpolation, $H$ extends to an element of $\mathbb B\big(\mathcal G^\si,\mathcal G^{\si-2}\big)$ for every $\si\in[0,2]$\,; in particular (abuse of notation) one has $H\in\mathbb B\big(\mathcal G^1,\mathcal G^{-1}\big)$\,.

\smallskip
Assume that we are given a unitary strongly continuous $1$-parameter group in $\H$\,, i.e. a family $\,W\equiv\big\{W(t)\mid t\in\R\big\}$ where each $W(t)$ is a unitary operator in $\H$\,, the mapping $\,t\to W(t)u\in\H$ is continuous for every $u\in\H$ and $W(s)W(t)=W(s+t)$\,, $\forall\,s,t\in\R$\,.
It is a standard fact that, if each $W(t)$ leaves $\mathcal G^1$ invariant, then (by duality and intepolation) one gets a $C_0$-group $W_\si$ in each Hilbert space $\mathcal G^\si$ for $\si\in[-1,1]$\,, with obvious coherence properties. In general, only $W_0\equiv W$ is unitary. We define $A$ to be the self-adjoint infinitesimal generator of $W$ (such that $W(t)=e^{itA}$\,, by Stone's Theorem). 

\begin{Definition}\label{classc1}
\begin{enumerate}
\item[(a)]
We say that $H$ is {\rm of class $C^1$ with respect to the group $W$} if $\,W(t)\mathcal G^1\subset\mathcal G^1$ for every $t\in\R$ and the map
\begin{equation*}\label{brima}
\R\ni t\to W_{-1}(-t)HW_1(t)\in\mathbb B\big(\mathcal G^1,\mathcal G^{-1}\big)
\end{equation*}
is differentiable. We write $H\in C^1\big(W;\mathcal G^1,\mathcal G^{-1}\big)$ and use the notation
\begin{equation*}\label{derivate}
K:=\frac{d}{dt}\Big\vert_{t=0}\big[W_{-1}(-t)HW_1(t)\big]\in\mathbb B\big(\mathcal G^1,\mathcal G^{-1}\big)\,.
\end{equation*} 
\item[(b)] We say that $A$ is {\rm weakly conjugate to $H$} if $\,H\in C^1\big(W;\mathcal G^1,\mathcal G^{-1}\big)$ and $K>0$\,, meaning that
\begin{equation*}\label{ozitiv}
\<u,Ku\>>0\,,\quad\forall\,u\in\mathcal G^1\setminus\{0\}\,.
\end{equation*}
\end{enumerate}
\end{Definition}

\begin{Remark}\label{observ}
{\rm There is a way to give $K$ a rigorous meaning as a commutator $K=[H,iA]$\,. Defining $A_1$ to be the infinitesimal generator of the $C_0$-group $W_1$ (it will be a closed operator in $\mathcal G^1$ with domain $D(A_1)$ and a restriction of $A$) one has
\begin{equation}\label{niu}
\<u,Kv\>=\big\<Hu,iA_1v\big\>-\big\<iA_1u,Hv\big\>\,,\quad\forall\,u,v\in D(A_1)
\end{equation}
and the condition $\,H\in C^1\big(W;\mathcal G^1,\mathcal G^{-1}\big)$ is equivalent to the continuity of the sesquilinear form \eqref{niu} with respect to the topology of $\,\mathcal G^1\times\mathcal G^1$.
}
\end{Remark}

To make our method work, one needs an extra requirement. First define the new Hilbert space $\dot\K$ to be the completion of $\mathcal G^1$ in the norm
\begin{equation}\label{sancheliu}
\p\!u\!\p_{\dot\K}\,:=\<u,Ku\>^{1/2}.
\end{equation}
The space $\dot\K$ should be seen a the "homogeneus Sobolev space" associated to $K^{1/2}$\,; it contains $\mathcal G^1$ continuously and densely. Note that in general (i.e.\! when $K$ does not have stronger positivity properties) $\dot\K$ and $\H$ are not comparable.

\begin{Remark}\label{obserf}
{\rm The topological anti-dual $\dot\K^*\hookrightarrow\mathcal G^{-1}$ of $\dot\K$ can be identified with the completion of $K\mathcal G^1\subset\mathcal G^{-1}$ in the Hilbert norm $\p\!w\!\p_{\dot\K^*}:=\<K^{-1}w,w\>^{1/2}$. Then $K$ extends to a unitary operator $\dot\K\mapsto\dot\K^*$.
}
\end{Remark}

Assume now that the $C_0$-group $W_1$ extends to a $C_0$-group $\dot W_1$ in $\dot\K$\,. This is equivalent to each $W_{-1}(t)\in\mathbb B\big(\mathcal G^{-1}\big)$ leaving $\dot\K^*$ invariant and thus yielding a $C_0$-group $\dot W_{-1}$ in $\dot\K^*$. Then it makes sense to ask $K\in C^1\big(W;\dot\K,\dot\K^*\big)$\,; this simply means that the mapping $\,t\to \dot W_{-1}(-t)K\dot W_{1}(t)\!\in\mathbb B\big(\dot\K,\dot\K^*\big)$ is differentiable and it is equivalent to the fact that the second commutator $[K,iA]$\,, first defined as a sesquilinear form on the domain of $A$ in $\dot\K$\,, extends continuously to $\dot\K\times\dot\K$\,. 

\smallskip
Now we can state the abstract spectral result, using the notions and notations above.

\begin{Theorem}\label{eorem}
If $A$ is weakly conjugate to the self-adjoint operator $H$ and the commutator $K\equiv[H,iA]$ belongs to $C^1\big(W;\dot\K,\dot\K^*\big)$\,, then $H$ has purely absolutely continuous spectrum.
\end{Theorem}

\subsection{Dilations}\label{vitanidis}

We prefer to introduce the dilations on the Carnot group $\G$ in terms of the action of the additive group $(\R,+)$\,. Denoting by ${\Aut}(\g)$ the space of automorphisms of the Lie algebra $\g=\mathfrak v_1\oplus\dots\oplus \mathfrak v_r$\,, we set $\,\mathfrak{dil}:\R\to{\Aut}(\g)$ by
\begin{equation*}\label{dila}
\mathfrak{dil}_t\big(Z_{(1)},Z_{(2)},\dots,Z_{(r)}\big):=\big(e^t Z_{(1)},e^{2t}Z_{(2)}\dots,e^{rt}Z_{(r)}\big)\,,\quad t\in\R\,,\ Z_{(k)}\in\mathfrak v_k\,,\ 1\le k\le r\,.
\end{equation*}
The elements $X_1,\dots,X_{m_1}$ are homogeneous of degree 1\,; more generally one has 
\begin{equation*}\label{stoarpha}
\mathfrak{dil}_t(X_j)=e^{\nu_j t} X_j\,,\quad t\in\R\,,\ 1\le j\le m\,.
\end{equation*}
Then we transfer the dilations to the group by 
\begin{equation*}\label{ansfer}
{\sf dil}_t(x):={\sf exp}\big[\mathfrak{dil}_t({\sf log}\,x)\big]\,,\quad x\in\G\,,\ t\in\R\,.
\end{equation*}
One then has natural actions of $(\R,+)$ on spaces of functions or distributions on $\g$ or on $\G$\,, of the form $f\to f\circ\mathfrak{dil}_t$ or $\,u\to u\circ{\sf dil}_t$\,. For the sake of unitarity, for $u\in L^2(\G)$ we set
\begin{equation*}\label{responsabil}
\big[{\sf Dil}(t)u\big](x):=e^{\frac{Mt}{2}}\big(u\circ{\sf dil}_t\big)(x)=e^{\frac{Mt}{2}}u\big({\sf dil}_t(x)\big)\,,
\end{equation*}
getting a unitary strongly continuous $1$-parameter group. Its infinitesimal generator, given by Stone's Theorem, is defined by ${\sf Dil}(t)=e^{itA}$\,. It is easy to see that it is given on $C^\infty_0(\G\setminus\{\e\})$ by
\begin{equation}\label{egat}
(Au)(x)=\sum_{j=1}^m \nu_j {\sf x}_j\big[(-iX_j)u\big](x)-\frac{iM}{2}u(x)\,.
\end{equation}
For the computation, besides the Dominated Convergence Theorem, one needs to note that 
\begin{equation*}\label{need}
{\sf dil}_t\Big(\exp\big[\sum_{j=1}^m{{\sf x}_jX_j}\big]\Big)=\exp\big[\sum_{j=1}^m{e^{\nu_j t}{\sf x}_jX_j}\big]
\end{equation*} 
and then use $\,X_j({\sf x}_ju)=u+{\sf x}_jX_ju$ and $\,\sum_{j=1}^m\nu_j=\sum_{k=1}^r km_k=M$. The formula \eqref{egat} can be written as
\begin{equation*}\label{generatoru}
A=\frac{1}{i}\sum_{j=1}^m \nu_j{\sf x}_j X_j+\frac{M}{2i}{\rm id}=\frac{1}{2i}\sum_{j=1}^m \nu_j\big({\sf x}_j X_j+X_j{\sf x}_j\big)=\frac{1}{2i}(\mathbf E-\mathbf E^*)=\im\mathbf E\,,
\end{equation*}
in terms of {\it the Euler operator} $\,\mathbf E:=\sum_{j=1}^m \nu_j{\sf x}_j X_j$ that has alredy been used in Theorem \ref{maine}. The operators $\mathbf E$ and $A$ depend on the homogeneous structure of $\G$\,, but not on the chosen basis.

\subsection{Proof of Theorem \ref{main}}\label{intrebare}

{\bf 1.}
Associated to the (positive and injective) operator $-\Delta$ one has the chain of Sobolev spaces $\big\{\H^\si(\G)\mid\si\in\R\big\}$\,, where $\H^0(\G)=L^2(\G)$ and for $\si>0$ the space $\H^\si(\G)$ is the domain of $(-\Delta)^{\si/2}$ with the graph norm.
Since $V$ is bounded (relative boundedness with subunitary relative bound would be enough), the domain of $H_\alpha=(-\Delta)^{\alpha/2}+V$ is also the  Sobolev space $\H^\alpha(\G)$ and the graph norm of $H_\alpha$ is equivalent to the graph norm of $L^\alpha=(-\Delta)^{\alpha/2}$. From the identity 
\begin{equation}\label{obama}
{\sf Dil}(-t)(-\Delta){\sf Dil}(t)=e^{2t}(-\Delta)
\end{equation}
it follows immediatly that ${\sf Dil}(-t)L^\alpha{\sf Dil}(t)=e^{\alpha t}L^\alpha$ and that $\mathcal G^1\equiv\H^{\alpha/2}(\G)$ is left invariant by the unitary group ${\sf Dil}$\,; this is also true for the other Sobolev spaces. One has for every $t\in\R$
\begin{equation*}\label{trump}
{\sf Dil}(-t)V{\sf Dil}(t)=V\circ{\sf dil}_{-t}\,.
\end{equation*}
Then, straightforwardly, $\,H_\alpha\in C^1\big({\sf Dil};\H^{\alpha/2}(\G),\H^{-\alpha/2}(\G)\big)$ and one has 
\begin{equation*}\label{gomutator}
K_\alpha\equiv[H_\alpha,iA]=\alpha L^\alpha-\mathbf EV\in\mathbb B\big[\H^{\alpha/2}(\G),\H^{-\alpha/2}(\G)\big]\,.
\end{equation*} 
Actually $K_\alpha$ even belongs to $\mathbb B\big[\H^\alpha(\G),L^2(\G)\big]$\,, which is stronger, but we are not going to need this.

\smallskip
\noindent
{\bf 2.}
Recalling the notation $\kappa_{2\alpha}=\big\Vert\lfloor x\rfloor^{-\alpha}(-\Delta)^{-\alpha/2}\big\Vert_{L^2}$\,, an important tool subsequently will be the Hardy type inequality \cite{CCR}
\begin{equation}\label{hardy}
\kappa_{\alpha}^2\int_\G|(L^{\alpha/2} u)(x)|^2 dx\ge\int_\G\frac{|u(x)|^2}{\lfloor x\rfloor^{\alpha}}dx\,,
\end{equation}
valid for every $u\in C^\infty_0(\G)$ if $\alpha\in(0,M/2)$\,. It can be rephrased as
\begin{equation}\label{reads}
\big\<L^\alpha u,u\big\>\equiv\<u,u\>_{\dot\H^{\alpha/2}}\ge\kappa_{\alpha}^{-2}\big\< \lfloor x\rfloor^{-\alpha}u,u\big\>\,,\quad\forall\,u\in C^\infty_0(\G)\,.
\end{equation}
But this is easily extended to $u\in\H^{\alpha/2}(\G)\equiv\mathcal G^1$\,, since $C^\infty_0(\G)$ is dense both in this Sobolev space and in the domain of the operator of multiplication by the function $x\to\lfloor x\rfloor^{-\alpha}$.
By our assumption (i) and by the extension of \eqref{reads}, if $u\in\H^{\alpha/2}(\G)\setminus\{0\}$ then
$$
\begin{aligned}
\big\<K_\alpha u,u\big\>&=\big\<[\alpha L^\alpha-\mathbf EV]u,u\big\>\\
&\ge \alpha\big\<L^\alpha u,u\big\>-\big\<|\mathbf EV|u,u\big\>\\
&\ge \alpha\big\<L^\alpha u,u\big\>-\Big[\alpha\kappa_{\alpha}^{-2}-\epsilon\Big]\big\<\lfloor x\rfloor^{-\alpha}u,u\big\>\\
&\ge\epsilon\big\<\lfloor x\rfloor^{-\alpha}u,u\big\> >0\,.\\
\end{aligned}
$$
So we showed that $A$ is weakly conjugate to $H$.

\smallskip
\noindent
{\bf 3.}
Let us focus now on the space $\dot\K\equiv\dot\K_\alpha$\,. We define {\it homogeneous Sobolev spaces} $\big\{\dot\H^\si(\G)\mid\si\in\R\big\}$\,. For $\si\ge 0$\,, it is the Hilbert space obtained by completing $\H^{\si}(\G)$ (or $C_0^\infty(\G)$) in the norm
\begin{equation*}\label{homosob}
\p\!u\!\p_{\dot\H^\si}:=\p\!(-\Delta)^{\si/2}u\!\p_{L^2}=\p\! L^{\si}u\!\p_{L^2}.
\end{equation*}
Thus $L^2(\G)=\H^0(\G)=\dot\H^0(\G)$\,. This second chain of Hilbert spaces is no longer increasing. Using Hardy's inequality one proves immediately that $\dot{\mathcal K}_\alpha\cong\dot{\H}^{\alpha/2}(\G)$ and $\dot{\mathcal K}_\alpha^*\cong\dot{\H}^{-\alpha/2}(\G)$ (same vector spaces with equivalent norms)\,. Then \eqref{obama} shows that ${\sf Dil}$ does extend to a $C_0$-group in $\dot{\H}^{\alpha/2}(\G)$\,, and actually to all the other homogeneous Sobolev spaces.
The remaining condition $K_\alpha\in C^1\big(W;\dot\K_\alpha,\dot\K_\alpha^*\big)$ is equivalent to $K_\alpha\in C^1\big(W;\dot{\H}^{\alpha/2}(\G),\dot{\H}^{-\alpha/2}(\G)\big)$ and boils down finally to the fact that the second commutator 
\begin{equation*}\label{siecund}
[K_\alpha,iA]=\alpha^2 L^\alpha+\mathbf E(\mathbf E V)
\end{equation*} 
is dominated by $C L^\alpha$ for some positive constant $C$\,. The assumption (ii) and the Hardy inequality, applied again, finish the proof.

\section{Other spectral results}\label{umilire}

Our proof of Theorem \ref{maine} relies on the abstract weakly conjugate method, on the dilation structure of the Carnot group and on the Hardy type estimate. We are going to indicate three other results, relying on different but connected inequalities. Higher order operators with variable coefficients in the principal part will hopefully be the topic of a subsequent paper.

\smallskip
It is known that there is a homogeneous quasi-norm $\,\rho:\G\to[0,\infty)$\,, continuous and smooth outside $\{\e\}$\,, such that, for some positive constant $C_M$\,,  the function $C_M\rho^{2-M}$ is a fundamental solution of $-\Delta$\,.

\begin{Theorem}\label{main}
Assume that $M\ge 3$\,. Let $\,V:\G\to\mathbb R$ be a bounded function of class $C^2(\G\setminus\{\e\})$\,,  with $\mathbf EV$ bounded, such that
\begin{enumerate}
\item[(i)] $\,|\mathbf E V|\le\Big[\frac{(M-2)^2}{2}-\epsilon\Big]\frac{(\nabla\rho)^2}{\rho^2}$\, for some $\epsilon>0$\,,
\item[(ii)] $\,|\mathbf E(\mathbf E V)|\le \,B\,\frac{(\nabla\rho)^2}{\rho^2}\,$ for some positive constant $B$\,.
\end{enumerate}
Then the self-adjoint operator $H:=-\Delta+V$ has purely absolutely spectrum.
\end{Theorem}

\begin{proof}
The proof is similar with the one above, replacing \eqref{reads} by the inequality  \cite{D'A,Ko,Li}
\begin{equation*}\label{hardi}
-\Delta\ge \Big(\frac{M-2}{2}\Big)^2\frac{(\nabla\rho)^2}{\rho^2}\,,
\end{equation*}
in which the constant is optimal.
\end{proof}

\begin{Remark}\label{duduta}
{\rm It is known that the homogeneous quasi-norms are equivalent. However, the constants are important in the conditions (i) of Theorems \ref{maine} and \ref{main}.
It is useful to note that the extra factor $\nabla\rho$ is bounded.
}
\end{Remark}

For the next result we need some extra notations, connected to the decomposition \eqref{lie}: For $X\in\g$ we set $X=\tilde X+X'$, where $\tilde X$ (the horizontal part) is the component of $X$ in $\mathfrak v_1$ and $X'$ its component in $\mathfrak v':=\mathfrak v_2\oplus\dots\oplus\mathfrak v_r$\,. Then, for $x\in\G$ one sets $\tilde x:={\sf exp}\big[\widetilde{{\sf log}\,x}\big]$\,. In terms of coordinates, if $\,x={\sf exp}\big[\sum_{j=1}^m{\sf x}_jX_j\big]$\,, then $\,\tilde x={\sf exp}\big[\sum_{j=1}^{m_1}{\sf x}_jX_j\big]$\,; thus $\,\tilde x\equiv({\sf x}_1,\dots,{\sf x}_{m_1})$\,. By $\,|\tilde x_{(1)}|:=\big(\sum_{j=1}^{m_1}|{\sf x}_j|^2\big)^{1/2}$ we denote the Euclidean norm.

\begin{Theorem}\label{naim}
For $m_1\ge 3$\,, let $\,V:\G\to\mathbb R$ be a bounded function of class $C^2(\G\setminus\{\e\})$\,, with $\mathbf EV$ bounded, such that
\begin{enumerate}
\item[(i)] $\,|\mathbf E V|\le\Big[\frac{(m_1-2)^2}{2}-\epsilon\Big]\frac{1}{|\tilde x|^{2}}$\, for some $\epsilon>0$\,,
\item[(ii)] $\,|\mathbf E(\mathbf E V)|\le \frac{B}{|\tilde x|^{2}}\,$ for some positive constant $B$\,.
\end{enumerate}
Then the self-adjoint operator $H:=-\Delta+V$ has purely absolutely spectrum.
\end{Theorem}

\begin{proof}
Instead of \eqref{hardy} one uses another Hardy type inequality \cite{RS2}:
\begin{equation*}\label{littlewood}
\int_\G|(\nabla u)(x)|^2 dx\ge\Big(\frac{m_1-2}{2}\Big)^2\int_\G\frac{|u(x)|^2}{|\tilde x|^2}dx\,.
\end{equation*}
Once again the constant is optimal.
\end{proof}

\begin{Remark}\label{dubita}
{\rm The two expressions $\frac{(M-2)^2}{2}\frac{(\nabla\rho)^2}{\rho^2}$ and $\frac{(m_1-2)^2}{2}\frac{1}{|\tilde x|^{2}}$ are incomparable in general. For instance, if $\,{\sf H}_1\equiv\R^3\ni({\sf x}_1,{\sf x}_2,{\sf t})$ is the simplest Heisenberg group, one has $m_1=2\,, m=3\,,M=4$\,, $\,X_1=\partial_1+2{\sf x}_2\partial_{\sf t}$\,, $\,X_2=\partial_2-2{\sf x}_1\partial_{\sf t}$ and 
\begin{equation*}\label{treipatati}
\frac{1}{\rho^2}=\frac{1}{\big[({\sf x}_1^2+{\sf x}_1^2)^2+{\sf t}^2\big]^{1/2}}\,\,,\quad\frac{(\nabla\rho)^2}{\rho^2}=\frac{{\sf x}_1^2+{\sf x}_2^2}{({\sf x}_1^2+{\sf x}_2^2)^2+{\sf t}^2}\,\,,\quad\frac{1}{|\tilde x|^{2}}=\frac{1}{{\sf x}_1^2+{\sf x}_2^2}\,.
\end{equation*}
Note that there is no decay in the variable ${\sf t}$ in the last expression. In this case $\frac{1}{\rho^2}\le\frac{1}{|\tilde x|^2}$ and $\frac{|\nabla\rho|^2}{\rho^2}\le\frac{1}{|\tilde x|^2}$\,, but the constants $\frac{(M-2)^2}{2}$ and $\frac{(m_1-2)^2}{2}$  (deduced from optimal constants in the Hardy inequalities) distroy a direct comparison.
}
\end{Remark}

\begin{Remark}\label{kombin}
{\rm The Heisenberg group allows a simple illustration of a phenomenon which is also present in other situations and which is connected to the independence of $\,|\tilde x|^{-2}$ on the variable $x'\in\mathfrak v':=\mathfrak v_2\oplus\dots\oplus\mathfrak v_r$\,. Assume for example that 
$$
V({\sf x}_1,{\sf x}_2,{\sf t})=\frac{\gamma\,U({\sf t})}{1+{\sf x}_1^2+{\sf x}_2^2}\,\,,\quad\gamma>0\,.
$$ 
Using the obvious formula $\,{\sf x}_1X_1+{\sf x}_2X_2={\sf x}_1\partial_{\sf x_1}+{\sf x}_2\partial_{\sf x_2}\,,$ one gets
$$
(\mathbf E V)({\sf x}_1,{\sf x}_2,{\sf t})=\frac{2\gamma}{1+{\sf x}_1^2+{\sf x}_2^2}\Big[{\sf t}U'({\sf t})-\frac{{\sf x}_1^2+{\sf x}_2^2}{1+{\sf x}_1^2+{\sf x}_2^2}U({\sf t})\Big]\,.
$$
Then the first condition (i) of Theorem \ref{naim} is fulfilled if $U$ and ${\sf t}U'$ are bounded and $\gamma$ is small enough. It is easy to check that the only extra requirement comming from condition (ii) is the boundedness of ${\sf t}^2U''$. Thus the behavior of $V$ in the variable ${\sf t}$ is very general and does not imply a decay of the function $U$ itself.
}
\end{Remark}

\begin{Remark}\label{kokombin}
{\rm Actually Theorems \ref{maine}, \ref{main} and \ref{naim} can be combined in a single one, gaining some generality; this can be easily seen from the proofs. For example, the conditions (i) from the three Theorems can be joint in the single requirement
\begin{equation*}\label{quirement}
|\mathbf E V|\le 2\th_1\kappa^{-2}_{2}\lfloor x\rfloor^{-2}+\theta_2\frac{(M-2)^2}{2}\frac{(\nabla\rho)^2}{\rho^2}+\theta_3\frac{(m_1-2)^2}{2}|\tilde x|^{-2}
\end{equation*}
for some positive constants $\th_1,\th_2,\th_3$ with $\th_1+\th_2+\th_3<1$\,. It is assumed that $M\ge 3$\,. For the condition (ii) the terms $2\kappa^{-2}_{2}\lfloor x\rfloor^{-2}$, $\frac{(M-2)^2}{2}\frac{(\nabla\rho)^2}{\rho^2}$ and $\frac{(m_1-2)^2}{2}|\tilde x|^{-2}$ can be combined with arbitrary (sufficiently large) constants. We kept the statements apart because they are easier to analyse individually.
}
\end{Remark} 

Finally, let us recall another Hardy inequality, on the $(2d+1)$-dimensional Heisenberg group ${\sf H}_d\equiv\R^{2d}\times\R\ni({\sf z},{\sf t})$\,. In this case the special homogeneous quasi-norm $\rho$ has the form
\begin{equation*}\label{solfundam}
\rho({\sf z},{\sf t})=\big(|{\sf z}|^4+{\sf t}^2\big)^{1/4}.
\end{equation*}
It has been shown in \cite{BCG} that 
\begin{equation*}\label{izabel}
(-\Delta)^{\alpha/2}\ge E_\alpha\rho^{-\alpha}
\end{equation*} 
if $\alpha\in(0,2d+2)$\,. It seems that the best constant $E_\alpha$ is unknown. Note that $\alpha\ne 2$ is permitted and that no $\nabla\rho$ factor is needed. Then one easily gets

\begin{Theorem}\label{remain}
Assume that $\alpha<2d+2$\,. Let $\,V:\G\to\mathbb R$ be a bounded function of class $C^2(\G\setminus\{\e\})$\,,  with $\mathbf EV$ bounded, such that
\begin{enumerate}
\item[(i)] $\,|\mathbf E V|\le\big(E_\alpha-\epsilon\big)\rho^2$\, for some $\epsilon>0$\,,
\item[(ii)] $\,|\mathbf E(\mathbf E V)|\le \,B\rho^{-\alpha}\,$ for some positive constant $B$\,.
\end{enumerate}
Then the self-adjoint operator $H:=-\Delta+V$ has purely absolutely spectrum.
\end{Theorem}

\bigskip
\noindent
{\bf Acknowledgements:} The author has been supported by Proyecto Fondecyt 1160359.


\bigskip
\bigskip
ADDRESS: Departamento de Matem\'aticas, Universidad de Chile, 

\smallskip
\quad\quad\quad\quad\quad\,Las Palmeras 3425, Casilla 653, Santiago, Chile.

\smallskip
\quad\quad\quad\quad\quad\,mantoiu@uchile.cl

\end{document}